\documentclass[A4paper,12pt]{article}
\usepackage{latexsym}
\usepackage{mathrsfs}
\usepackage{amssymb}
\usepackage{amscd}
\usepackage[dvips]{graphicx}                  

\newtheorem{theorem}{Theorem}
\newtheorem{corollary}{Corollary}

\newtheorem{proposition}{Proposition}
\newenvironment{proof}[1][Proof]{\textbf{#1.} }{\ \rule{0.5em}{0.5em}}
\long\def\symbolfootnote[#1]#2{\begingroup%
	\def\thefootnote{$\;$}\footnote[#1]{$^*$#2}\endgroup}
\begin{document}
	
	\title{On Banach and Kuratowski Theorem and generalized strong sequences}
	\author{Joanna Jureczko\footnote{Wroc\l{}aw University of Science and Technology, Poland,			
			e-mail: joanna.jureczko@pwr.edu.pl}}
	
	\maketitle
	
	\symbolfootnote[2]{Mathematics Subject Classification: 03E02, 03E05, 03E10, 03E20, 03E35, 03E52, 03E55, 03E57, 03E65, 54E52.} 
	
	\hspace{0.2cm}
	\symbolfootnote[3]{Keywords: \textsl{Banach and Kuratowski matrix, K-Lusin set,  strong sequences, calibre, boundedness, dominating number, generalized Baire space.}}
	
	\begin{abstract}
		In 1929, Banach and Kuratowski proved under CH a combinatorial theorem, which implies that there is not  a non-vanishing $\sigma$-additive finite measure on $\mathbb{R}$ which is defined for every set of reals.    
		In 2003 Bartoszy\'nski and Halbeisen proved that Banach and Kuratowski theorem  is equivalent to the existence of a K-Lusin set of the cardinality continuum an the existence of such sets is independent of $ZFC + \neg CH$. On the other hand in 1965 Efimov introduced the strong sequences method which used to prove some well-known theorems in dyadic spaces. The aim of this paper is to show that all this notions can be generalized, i.e. considered in the generalized Baire space and to show some equivalences among them.
		Moreover,  some applications in the direction of calibres, boundedness and partitions relations are also shown.
	\end{abstract}
	\maketitle
	
	\section{Introduction}
	Based on the paper \cite{BH} and author's papers  on strong sequences \cite{JJ1, JJ2, JJOpen}, this paper will be devoted to considerations on concepts known in the literature as BK-matrix, K-Lusin set and certain cardinal invariants.
	
	In \cite{BH} the authors recall the following theorem proved by Banach and Kuratowski in 1929.
	\\
	\\
		\textbf{Banach and Kuratowski Theorem}
	\textit{Under the assumption of CH, there is an infinite matrix $A^{i}_{k} \subseteq [0,1]$ (where $i,k \in \omega$) such that}
	\\	
	\textit{(i) for each $i \in \omega, [0,1] = \bigcup_{k \in \omega} A^{i}_{k}$,}
	\\	
	\textit{(ii) for each $i \in \omega, $ if $k \not = k'$ then $A^{i}_{k} \cap A^{i}_{k'} = \emptyset$,}
	\\	
	\textit{(iii) for every sequence $k_0, k_1, ..., k_i,... $ of $\omega$ the set $$\bigcap_{i \in \omega} (A^{i}_{0} \cup A^{i}_{1} \cup ... \cup A^{i}_{k_i})$$ is at most countable. }
	\\
	\\
	The matrix from the theorem above is called in the literature as a \textit{BK-matrix}.
	\\
	
	The existence of such a matrix under the assumption (CH) is the answer to the question whether  there exists a non-vanishing $\sigma$-additive finite measure on $\mathbb{R}$ which is defined for every set of reals. 
	Obviously, the answer is negative. The theorem quoted above has further consequences, which are largely contained in the paper \cite{BH} where the authors conduct their considerations in the Polish space ${^\omega}\omega$.
	
	In the current paper, we will generalize the concepts contained in \cite{BH} to the generalized Baire space ${^\kappa}\kappa$ for $\kappa>\omega$, and demonstrate connections between the concepts from \cite{BH} and strong sequences and certain cardinal invariants known in the literature. We will also refer to the results from \cite{JJOpen} and provide their generalizations. (We also give some corrections and further explanations concerned the results from \cite{JJOpen}).
	
	There are many papers in the literature concerning cardinal invariants in ${^\omega}\omega$, and this topic seems to be quite well-studied. Recently, there appear results concerning considerations in the generalized Baire space ${^\kappa}\kappa$ when $\kappa>\omega$ is regular. This notion was introduced in \cite{H}. It turns out that some of the notions considered in the case $\kappa = \omega$ can be easily generalized to the case $\kappa>\omega$  but not all of them, see e.g \cite{HS, FS}. Additional assumptions are then required, e.g., $\kappa = \kappa^{< \kappa}$. Interestingly, considerations in the case of $\kappa>\omega$ allow us to solve long-unsolved problems for $\kappa = \omega$, (e.g. Roitman's problem).
	The paper \cite{FS} also shows results for $\kappa>\omega$, which are false for $\kappa=\omega$.
	
	In 1995, in \cite{CS}  invariants in the space ${^\kappa}\kappa$ for $\kappa>\omega$ were considered for the first time. In subsequent papers on this topic, especially those co-authored by Shelah, the topic was gradually developed.
	
	An interesting paper in the literature is  \cite{KLLS}, in which the authors summarized the current knowledge (as of the year of publication) regarding research on ${^\kappa}\kappa$ for $\kappa>\omega$. That paper is the result of two seminars devoted to this topic: in Amsterdam in 2014 and in Hamburg in 2015.
		Other interesting papers on this subject are \cite{RS, RS1} and the continuation in \cite{FS}, in which one can also find a historical overview on this topic.
		
The strong sequences which are mentioned in the second part of the title of this paper were introduced in \cite{BE} as a useful method for proving a number of famous theorem in dyadic spaces (i.e. continuous images of the Cantor cube) were  adopted by M. Turza\'nski and successfully developed  in \cite{JJ1, JJ2, JJ3, JJ4, JJ5}. Especially, the results contained in \cite{JJ5} shows that strong sequences can be helpful for obtaining new inequalities between cardinal invariants.    
	
The notation used in this paper is standard for the fields of topology and sets theory. For definitions and fact not cited here, we refer to e.g. \cite{TJ} (set theory) and	\cite{RE} (topology).

	\section{Definitions and some auxiliary results}
	In the whole paper we assume $ZFC$ in every place whenever we do not assume otherwise.
	 
	\subsection{Generalized Baire space}
	Following \cite{KLLS}, we assume the following definition. Let $\kappa> \omega$ be a regular cardinal. We say that $(^{\kappa}\kappa, \tau)$ with the topology $\tau$ generated by the basic open sets of the form 
	$$[s] = \{f \in ^{\kappa}\kappa \colon f\upharpoonright |s| = s\}$$
	with $s \in \kappa^{\kappa^{< \kappa}}$
	is the \textit{generalized Baire space} and the topology $\tau$ is the \textit{bounded topology}.
	
	For the study of the generalized Baire space from a topological point of view the assumption that $\kappa^{<\kappa} = \kappa$  nearly always cannot be omitted, see \cite[II.2.1]{FHK}.

We add that a generalized Cantor space $2^\kappa$ with the bounded topology one defines analogously.
 
\subsection{Some cardinal invariants in $^{\kappa}\kappa$}

In ${^\kappa}\kappa$ we consider the following relation, if $f, g\in {^\kappa}\kappa$ then $$g \preceq_\kappa f \textrm{ iff }  |\{\alpha \in \kappa \colon g(\alpha) > f(\alpha)\}|< \kappa.$$
Because in the whole paper we will work in ${^\kappa} \kappa$ we hope that it will not to come to missunderstanding when we will write $\preceq$ instead of $\preceq_\kappa$.

Let $f, g \in {^\kappa} \kappa$. We say that $f$ dominates $g$ iff $g \preceq f$. We say that $f \in {^\kappa} \kappa$ \textit{dominates} $\mathcal{G} \subseteq {^\kappa} \kappa$ iff $g \preceq f$ for each $g \in \mathcal{G}$.

For any $\mathcal{F} \subseteq {^\kappa} \kappa$ let 
$$\lambda_\kappa (\mathcal{F}) = \min \{\tau \colon \forall_{g \in {^\kappa}\kappa}\  |\{f \in \mathcal{F} \colon f \preceq g\}| < \tau\},$$
$$\mathfrak{l}_\kappa = \min \{\lambda_\kappa (\mathcal{F}) \colon \mathcal{F} \subseteq {^\kappa}\kappa \wedge |\mathcal{F}| =2^\kappa\}.$$
It is obvious that if $\mathcal{F}$ has size $2^\kappa$, then $\kappa^+ \leqslant \lambda_\kappa (\mathcal{F}) \leqslant (2^\kappa)^+$.

A family $\mathcal{B}\subseteq {^\kappa} \kappa$ is called \textit{unbounded} if for all $g \in {^\kappa} \kappa$ there is $f \in \mathcal{B}$ such that $f \not \preceq g$.
 A family $\mathcal{D} \subseteq {^\kappa} \kappa$ is called \textit{dominating} if for all $g \in {^\kappa} \kappa$ there is $f \in \mathcal{D}$ such that $g \preceq f$. 
 
 Define
 $$\mathfrak{b}_\kappa = \min \{|\mathcal{B}| \colon \mathcal{B} \subseteq {^\kappa} \kappa \textrm{ is an unbounded family}\}.$$
 
 $$\mathfrak{d}_\kappa = \min \{|\mathcal{B}| \colon \mathcal{B} \subseteq {^\kappa} \kappa \textrm{ is a dominating family}\}$$
 and say that $\mathfrak{b}_\kappa$ is a \textit{bounding family} and $\mathfrak{d}_\kappa$ is a \textit{dominating family}.
  
By a generalized Cicho\'n diagram, see \cite{KLLS} we have
Obviously $$\mathfrak{b}_\kappa \leqslant \mathfrak{d}_\kappa$$
$$\kappa^+ \leqslant add(\mathcal{M}_\kappa) \leqslant \mathfrak{b}_\kappa \leqslant non(\mathcal{M}_\kappa) \leqslant cof(\mathcal{M}_\kappa) \leqslant 2^\kappa$$ and
$$\kappa^+ \leqslant add(\mathcal{M}_\kappa) \leqslant cov(\mathcal{M}_\kappa) \leqslant \mathfrak{d}_\kappa \leqslant cof(\mathcal{M}_\kappa) \leqslant 2^\kappa$$

	\subsection{Generalized BK-matrix}
	Let $\lambda, \eta$ and $\nu$ be infinite cardinals.
	A family $\mathcal{A} = \{\{A_{\alpha}^{\beta} \colon \alpha < \lambda\} \colon \beta < \eta \}$ of subsets of a set $S$ with $|S|= \lambda$ is called a \textit{($\lambda, \eta, \nu$)-$BK$-matrix} if $\eta \leqslant \lambda$ and $\nu < \lambda$ 
\begin{enumerate}
	\item  $ \bigcup_{\alpha < \lambda} A^{\beta}_{\alpha} = S$ for any $\beta < \eta$,

	\item $A^{\beta}_{\alpha} \cap A^{\beta}_{\alpha'} = \emptyset$ for any $\beta < \eta$ and $\alpha, \alpha' < \lambda$ with $\alpha \not = \alpha'$, 
	
	\item   $|\bigcap_{\beta < \eta} \bigcup_{\alpha =0}^{\gamma_\beta}A^{\beta}_{\xi_\beta}|\leqslant \nu$  for any sequence $(\gamma_\beta)_{\beta< \nu}$.
	\end{enumerate}
	
	If $\lambda = \nu^+$, then we call this matrix ($\lambda, \eta$)-$BK$-matrix  and if moreover $\lambda = \eta$ then we call it $\lambda$-$BK$-matrix.

	\subsection{Generalized strong sequence}
	
	Let $(S, r)$ be a set with a relation defined on $[S]^2$. Let $\lambda$ be a cardinal such that $|S| \geqslant \eta$.
	We say that a subset $A \subseteq S$ is \textit{$\eta$-directed}  if there exists $b \in S$ such that $arb$ for any $a \in A$.
	
	A sequence $(S_\xi, H_\xi)_{\xi < \Phi}$ is called \textit{$\eta$-strong sequence} if $|S_\xi|<\eta$ for any $\eta < \Phi$ and 
	\begin{enumerate}
		\item $S_\xi \cup H_\xi$ is $\eta$-directed,
		\item$S_\zeta \cup H_\xi$ is not $\eta$-directed, whenever $\xi < \zeta$.
	\end{enumerate}
	
	Observe that if $\kappa< \eta$ and if $(S_\xi, H_\xi)_{\xi < \Phi}$ is  $\eta$-strong sequence then  $(S_\xi, H_\xi)_{\xi < \Phi}$ is  $\kappa$-strong sequence. The converce implication is not true.
	\\
	
	If we consider the space ${^\kappa}\kappa$ we will interpret the above notion as follows. Let $\eta \leqslant |{^\kappa}\kappa|$. 
	A sequence $(S_\xi, H_\xi)_{\xi < \Phi}$ is called \textit{$\eta$-strong sequence} if $|S_\xi|< \eta$ for any $\xi < \Phi$ and 
	\begin{enumerate}
		\item  $H_\xi \subseteq \{g \in {^\kappa}\kappa \colon g \preccurlyeq f_\xi\}$ for $f_\xi \in {^\kappa}\kappa$ dominating $S_\xi$,
 		\item $\exists_{f_\zeta \in S_\zeta}$ $H_\xi \not\subseteq \{g \in {^\kappa}\kappa \colon g \preccurlyeq f_\zeta\}$, whenever $\xi < \zeta$.
	\end{enumerate}

	To avoid misunderstading we assume that $f_\xi$ in the above definition belongs to $S_\xi$\footnote{The reviewer of \cite{JJOpen} in MathScieRev changed the definition of strong sequences and consequently obtained fail theorems.}.

		In the generalized strong sequences we can associated the following cardinal invariant
	$$\hat{\mathfrak{s}}_\kappa = \sup\{\theta \colon \textrm{ there exists an $\eta$-strong sequence of length } \theta \textrm{ in } {^\kappa}\kappa \}.$$

			\subsection{Generalized K-Lusin set}
		
		A set $A \subseteq {^\kappa}\kappa$ is called \textit{$\kappa$-meager} if it is the union of at most $\kappa$ nowhere dense  (in the bounded topology) sets. 
		A set $L\subseteq {^\kappa}\kappa$ of size $2^{\kappa}$ is a \textit{$\kappa$-K-Lusin set} if $|X \cap K| \leqslant \kappa$ for every $\kappa$-meager set $K \subseteq {^\kappa}\kappa$.
		
		A set $K \subseteq {^\kappa}\kappa$ is said to be \textit{closed} iff for all $f \in K$ and for all $g \in {^\kappa}\kappa$  if $f \preceq g$ then $g \in K$.
		A closed set $K \subseteq {^\kappa}\kappa$ is \textit{$\kappa$-compact} iff there is a function $f \in {^\kappa}\kappa$ such that $K \subseteq \{g \in {^\kappa}\kappa \colon g \preceq f \}$.
		A set $X$ of size $2^{\kappa}$ is a \textit{$\kappa$-K-Lusin set} if $|X \cap K|\leqslant \kappa$ for every $\kappa$-compact set $K \subseteq {^\kappa}\kappa$.

		The next result was proved in the case $\kappa = \omega$ in \cite{BH}. 
		
		\begin{proposition}
			Every $\kappa$-Lusin set is a $\kappa$-K-Lusin set.
		\end{proposition}
		\begin{proof}
			The proof is an easy consequence of definitions of both notions because every $\kappa$-compact set is $\kappa$-meager. 
		\end{proof}
\\
		
Unfortunately, the reverse claim is not true, because they can exist $\kappa$-meager sets which are not $\kappa$-compact. Therefore, results from the literature that tell us when the Lusin set does not exist are not useful to us.
What we can definitely say is that, assuming $MA_\kappa +\neg CH_\kappa$ (i.e., $MA+\neg CH$ for $\kappa>\omega$), there is no $\kappa$-K-Lusin set with cardinality $2^\kappa$ and consequently nor notion in Theorem 1.
		
	\subsection{Generalized concentrated set}
	
	A set $A \subseteq {^\kappa}\kappa$ is called \textit{$\kappa$-concentrated} on a set $B \subseteq {^\kappa}\kappa$ if  every open set (in the bounded topology) $G \subseteq {^\kappa}\kappa$ such that $B \subseteq G$ contains all but $\kappa$ many elements of $A$.

	\section{Main results}
	
	\begin{theorem}  Let $\kappa > \omega$ be a regular cardinal such that $\kappa = \kappa^{<\kappa}$ and let $\lambda =  cf(2^\kappa)$. Let $X={^\kappa}\kappa$. Then the following are equivalent:
\begin{itemize}
	\item [(i)] there exists a $\kappa^+$-strong sequence $(S_\xi, H_\xi)_{\xi < \lambda}$ in $X$ such that $|H_\xi|\leqslant \kappa$ for each $\xi < \lambda$; 
\item [(ii)] there exists a $\kappa$-BK matrix in $X$;
\item [(iii)]  $\mathfrak{l}_\kappa = \kappa^+$;
\item [(iv)] There exists a $\kappa$-concentrated set of cardinality $\lambda$ in $X$;
\item [(v)] there exists a $\kappa$-K-Lusin set of cardinality $\lambda$ in $X$.
\end{itemize}
\end{theorem}
\begin{proof}
	(In the proof we use some ideas from \cite{BH}).
	
$[(i)] \Rightarrow [(ii)]$
Let $(S_\xi, H_\xi)_{\xi < \lambda}$ such that $|H_\xi|\leqslant \kappa$ for each $\xi < \lambda$ be a $\kappa^+$-strong sequence. By the definition of $\kappa^+$-strong sequence we have that $S_\zeta \cup H_\xi$ is not $\kappa^+$-directed. Hence for each $\xi< \zeta$ there exists  $f_\xi \in S_\xi$    such that $\{f_\zeta\} \cup H_\xi$ is not $\kappa^+$-directed  whenever $\xi<\zeta<\lambda$.
Let 
$$H=\{f_\xi \colon f_\xi \in S_\xi, \xi< \lambda\}.$$ 
Let $\Lambda \colon \lambda \to H$ be a bijection such that $\Lambda(\xi) = f_\xi$, $\xi < \lambda$.
For each $\alpha < \kappa$  define
$A^\alpha_\eta \in P(\lambda)$ as follows
$$\xi \in A^\alpha_\eta \textrm{ iff } \eta = \Lambda^\alpha(\xi).$$
Obviously, the sets $A^\alpha_\eta$, $\alpha, \eta< \kappa$ fulfills (1-2) of definition of $\kappa$-BK-matrix. To complete this part of the proof it is enough to show that (3) of definition of $\kappa$-BK-matrix is also fulfilled. 

For this purpose, choose arbitrary sequence 
$(\eta_\beta)_{\colon \beta < \kappa}$ of $\kappa$ and consider a set
 $$\bigcap_{\beta < \kappa}\bigcup_{\gamma < \eta_\beta} A^\beta_\gamma.$$
 If $\bigcap_{\beta < \kappa}\bigcup_{\gamma < \eta_\beta} A^\beta_\gamma = \emptyset$ then (3) is fulfilled.
 
 If $\bigcap_{\beta < \kappa}\bigcup_{\gamma < \eta_\beta} A^\beta_\gamma\not = \emptyset$ then take arbitrary $\zeta \in\bigcap_{\beta < \kappa}\bigcup_{\gamma < \eta_\beta} A^\beta_\gamma$. Then $\zeta \in \bigcup_{\gamma < \eta_\beta} A^\beta_\gamma$ for each $\beta < \kappa$. It means that $\Lambda^\beta(\xi) \leqslant \eta^\beta$ for each $\beta < \kappa$. Since $\Lambda(\xi) = f_\xi$ and $f_\xi \in H_\xi$ and $|H_\xi|\leqslant \kappa$, we have our claim. 

$[(ii)] \Rightarrow [(i)]$ Let $\{A^\alpha_\eta \in P(\lambda)$, for each $\alpha, \eta < \kappa\}$ be a $\kappa$-BK-matrix. Let $\mathcal{F} \subseteq {^\kappa}\kappa$ be the family of all functions such that $\bigcap_{\alpha< \kappa} A^\alpha_{f(\alpha)} \not = \emptyset$.
Obviously $|\mathcal{F}| = 2^\kappa$. 

Take arbitrary sequence $(\eta_\beta)_{\beta< \alpha}$ of $\kappa$. By (3)  of the definition of $\kappa$-BK-matrix
$$|\bigcap_{\beta < \kappa}\bigcup_{\gamma < \eta_\beta} A^\beta_\gamma|\leqslant \kappa.$$
Hemce, for each $g \in {^\gamma}\gamma$  such that $g(\alpha) = \eta_\beta$ 
$$|\{f \in \mathcal{F} \colon f \preccurlyeq g\}|\leqslant \kappa.$$
Hence $\lambda(\mathcal{F}) = \kappa^+$.

$[(iii)] \Leftrightarrow [(iv)]$ straightforward the definitions.

$[(iv)] \Rightarrow [(v)]$
Let $D \subseteq \lambda$ be a dense set of the cardinality $\kappa$. Take a homeomorphism $\varphi \colon \lambda \setminus D \to {^\kappa}\kappa$. Then for each $\kappa$-compact set $U \subseteq {^\kappa}\kappa$, $\varphi^{-1}[U]$ is $\kappa$-compact. and $\kappa^+ \setminus \varphi^{-1}[U]$ is an open set containing $D$. For any $A \subseteq \lambda$ of cardinality $\eta$ if $\varphi[A]$ is $\kappa$-concentrated in $D$ then $\varphi[A]$ is a $\kappa-K-$Lusin set and $|\varphi[A]|=\eta$.

$[(v)] \Rightarrow [(iv)]$
Let $D \subseteq \lambda$ be a dense set of cardinality $\kappa$. Take a homeomorphism $\varphi \colon \lambda \setminus D \to {^\kappa}\kappa$.
Let $B$ be a $\kappa$-K-Lusin set of the cardinality $\lambda$. Then $\varphi^{-1}[B]$ is $\kappa$-concentrated on $D$ and $\varphi^{-1}[B]$ has the cardinality $\lambda.$

$[(v)] \Rightarrow [(ii)]$
Let $A \subseteq {^\kappa}\kappa$ be a $\kappa$-K-Lusin set of cardinality $\lambda$. 
Enumerate
$$A = \{f_\xi \in {^\kappa}\kappa \colon \xi < \lambda\}.$$
Let $K_0 \subseteq {^\kappa}\kappa$ be a $\kappa$-compact set. By the definition of $\kappa$-K-Lusin set there exists $f_0 \in {^\kappa}\kappa$ such that $K_0 \subseteq \{g \in {^\kappa}\kappa \colon g\preccurlyeq f_0\}.$
Let $S_0 = \{f_0\}$ and $H_0 = A\cap K_0$. Since $A$ is a $\kappa$-K-Lusin set, $|H_0| \leqslant \kappa$. Let $(S_0, H_0)$ be the first pair of a $\kappa^+$-strong sequence $(S_\xi, H_\xi)_{\xi < \lambda}$. 

Assume that we have defined the $\kappa^+$-strong sequence $(S_\xi, H_\xi)_{\xi< \eta}$ for some $\eta < \lambda$ such that $S_\xi = \{f_\xi\}$ and $H_\xi = A\cap K_\xi$, where $K_\xi \subseteq \{g \in {^\kappa}\kappa \colon g\preccurlyeq f_\xi\}$ is $\kappa$-compact. Since $\eta < \lambda$ then 
$$A \setminus \bigcup\{H_\xi \colon \xi < \eta\} \not = \emptyset.$$  
Hence we can continue the construction of the $\kappa^+$-strong sequence $(S_\xi, H_\xi)_{\xi < \lambda}$.

Case 1. $\eta = \xi+1$. 
Let $f_\eta = A \setminus \bigcup\{H_\xi \colon \xi < \eta\}$ be such that 
$\{f_\xi\}\cup\{f_\eta\}$ is not $\kappa^+$-directed whenever $\xi < \eta$.
Let $K_\eta \subseteq \{g \in {^\kappa}\kappa \colon g\preccurlyeq f_\eta\}$ be a $\kappa$-compact set.
Let $S_\eta  = \{f_\eta\}$. By choosing $f_\eta$ we have that $S_\eta \cup H_\xi$ is not $\kappa^+$-directed for any $\xi < \eta$. 
Let $H_\eta = A \cap K_\eta$. Since $A$ is a $\kappa$-K-Lusin set, $|H_\eta|\leqslant \kappa$. Let $(S_\eta, H_\eta)$ be the next pair of the $\kappa^+$-strong sequence.

Case 2. $\eta$ is limit. Then $S_\eta = \bigcap_{\xi< \eta}S_\xi$ and $H_\eta = \bigcap_{\xi < \eta}H_\xi$. Let $(S_\eta, H_\eta)$ be the next pair of the $\kappa^+$-strong sequence $(S_\xi, H_\xi)_{\xi < \lambda}$.

Hence we have defined the $\kappa^+$-strong sequence $(S_\xi, H_\xi)_{\xi < \lambda}$.
\end{proof}

\begin{corollary}
	Assume GCH.	Let $\kappa > \omega$ be a regular cardinal such that $\kappa = \kappa^{<\kappa}$. Let $X={^\kappa}\kappa$. Then the following are equivalent:
	\begin{itemize}
		\item [(i)] there exists a $\kappa^+$-strong sequence $(S_\xi, H_\xi)_{\xi < \lambda}$ in $X$ such that $|H_\xi|\leqslant \kappa$ for each $\xi < 2^\kappa$; 
		\item [(ii)] there exists a $\kappa$-BK matrix in $X$;
		\item [(iii)]  $\mathfrak{l}_\kappa = 2^\kappa$;
		\item [(iv)] There exists a $\kappa$-concentrated set of cardinality $2^\kappa$ in $X$;
		\item [(v)] there exists a $\kappa$-K-Lusin set of cardinality $2^\kappa$ in $X$.
	\end{itemize}
\end{corollary}

	\section{Applications and further results}

	\subsection{Generalized independent families and strong sequences}
One of the most known family which is known in the literature is an independent family. As we assume that it is the folklore we only add that this notion has been generalized  in \cite{SE, WH}.  We recall here its definition. 
\\

	A family $\mathcal{I} = \{\{I_{\alpha}^{\beta} \colon \alpha < \lambda_\beta\} \colon \beta < \eta \}$ of subsets of a set $S$ with $|S|= \lambda$ is called a \textit{($\gamma, \mu$)-independent family} if
\begin{enumerate}
	\item  $ \bigcup_{\alpha < \lambda_\beta} I^{\beta}_{\alpha} = S$ for any $\beta < \eta$ and $\lambda_\beta \geqslant 2$ for each $\beta < \eta$,
	
	\item $I^{\beta}_{\alpha} \cap I^{\beta}_{\alpha'} = \emptyset$ for any $\beta < \eta$ and $\alpha, \alpha' < \lambda$ with $\alpha \not = \alpha'$, 
	
	\item   $|\bigcap_{\beta \in J}I^\beta_{\delta(\beta)}| \geqslant \mu$ for each $J \in [\eta]^\gamma$ and each $\delta \in \Pi_{\beta \in J}\lambda_\beta$.
\end{enumerate}
If $\lambda_\beta = \lambda$ for all $\beta < \eta$, then $\mathcal{I}$ is called a $(\gamma,\mu, \lambda )$-independent family on $S$. 
\\ 

Comparing the above definition with the  generalized BK-matrix definition in Section 2 we come to the conclusion that if we assume that the intersection in (3) in 2.3 is non-empty (this suggests the original proof by Banach and Kuratowski) and that the intersection in (3) above have the cardinality equal to $\mu$ the we have the following corollary. \footnote{Unfortunately \cite[Proposition 4.2]{JJOpen} and the information on p. 728 in \cite{JJOpen} and \cite[Corollary 5.3 (2)]{JJOpen} are false}.
\begin{corollary}
	Let $\kappa > \omega$ be a regular cardinal such that $\kappa = \kappa^{<\kappa}$ and let $\lambda = cf(2^\kappa)$. Let $S\subseteq {^\kappa}\kappa$ be a set with $|S| = \lambda$.
	If there exists a $(\kappa, \mu, \kappa)$-independent family on $S$ and $\mu \leqslant \kappa$ then there exists a $\kappa$-BK-matrix on the same set $S$.
	If moreover $\lambda$ is regular then the existence of a $(\kappa, \mu, \kappa)$-independent family on $S$ and $\mu \leqslant \kappa$ implies the existence of a $\kappa^+$-strong sequence in $S$ of length $\lambda$.  
\end{corollary} 

Other connections between generalized independent sets and strong sequences one can also be found  in \cite{JJ1}.

	\subsection{Calibres and strong sequences} In the literature there are known two notions which are cardinal invariants: calibre and precalibre. However they are generally different, there are spaces for which the are equal, for example in compact spaces. 
	In this part we show the relation between calibres and strong sequences in ${^\kappa}\kappa$.
	
	Let $(X, r)$ be the set with the relation $r$ on $[X]^2$. 
	Let $\lambda$ and $\eta$ be cardinals with $\omega  \leqslant \eta \leqslant \lambda$. $X$  is called to have a $(\lambda, \eta, \alpha)$-calibre if for any set $A \in [X]^\lambda$ of the cardinality $\lambda$ there exists a subset $B \subseteq A$ of  the cardinality  $\eta$ which is $\alpha$-directed. If  moreover $\lambda=\eta$ then we say that $X$ has a $(\lambda, \alpha)$-calibre and if moreover $\alpha = \omega$ then we say that $X$ has a $\lambda$-calibre.
	
\begin{theorem}  
	Let $\eta, \mu, \alpha$ be cardinals such that $\alpha \leqslant \mu < \eta$ and $\eta$ be regular. Let $(X, r)$ be a set  of cardinality at least $\eta$.
	If $X$ has no $(\eta, \mu, \alpha)$-calibre then there exists an $\alpha$-strong sequence $(S_\xi, H_\xi)_{\xi< \eta}$ such that $|H_\xi|< \mu$. 
\end{theorem}  

\begin{proof}
	Let $A \subseteq X$ be a set of cardinality $\eta$. Since there exists a $(\eta, \mu \alpha)$-calibre then each  $B\subseteq A$ which is $\alpha$-directed has cardinality less than $\mu$.
	
	Let $f_0 \in A$ be an arbitrary element and let $H_0 \subseteq A$ be a maximal $\alpha$-directed subset dominated by $f_0$. Let $S_0 = \{f_0\}$ and $H_0 = \{f \in A \colon frf_0\}$.  
	 Let $(S_0,H_0)$ be the first pair of the $\alpha$-strong sequence $(S_\xi, H_\xi)_{\xi< \eta}$.
	 
	Assume that we have defined the $\alpha$-strong sequence $(S_\xi, H_\xi)_{\xi< \zeta}$ such that $|H_\xi|< \mu$ for $\zeta<\eta$ such that Let $$S_\xi = \{f_\xi\}\textrm{ and }H_\xi = \{f \in A \setminus \bigcup\{H_\gamma \colon \gamma < \xi\} \colon frf_\xi\}.$$ 
	Since $\zeta < \eta$ then 
	$$A \setminus \bigcup\{H_\xi \colon \xi < \eta\} \not = \emptyset.$$  
	Hence we can continue the construction of the $\alpha$-strong sequence $(S_\xi, H_\xi)_{\xi< \eta}$.
	
	Consider two cases
	
	Case 1. $\zeta = \xi+1$. 
	Let $f_\zeta \in A \setminus \bigcup\{H_\xi \colon \xi < \eta\}$ is such that $\{f_\zeta\} \cup H_\xi$ is not $\alpha$-directed for each $\xi< \zeta$. Let $$S_\zeta = \{f_\zeta\} \textrm{ and } H_\zeta=\{f \in A\setminus\bigcup\{H_\xi \colon \xi < \zeta\} \colon frf_\zeta\}.$$
	Then $(S_\zeta, H_\zeta)$ be the next pair of the $\alpha$-strong sequence. 
	
	Case 2. $\zeta$ is limit then we take $S_\zeta = \bigcap_{\xi< \zeta}S_\xi$ and $H_\zeta = \bigcap_{\xi < \zeta} H_\xi$.
	Then $(S_\zeta, H_\zeta)$ be the next pair of the $\alpha$-strong sequence. 
	
	Hence we have constructed $\alpha$-strong sequence $(S_\xi, H_\xi)_{\xi < \eta}$. 
\end{proof}

\begin{corollary}
	Let $\kappa > \omega$ be a regular cardinal such that $\kappa = \kappa^{<\kappa}$ and let $\lambda = cf(2^\kappa)$. Let $\eta, \mu, \kappa^+$ be cardinals such that $\kappa^+ \leqslant \mu < cf(2^\kappa)$. Let $(X, \preceq)$ be a set such that $X \subseteq {^\kappa}\kappa$ of cardinality at least $\eta$.
	If $X$ has no $(\eta, \mu, \alpha)$-calibre then there exists an $\alpha$-strong sequence $(S_\xi, H_\xi)_{\xi< \eta}$ such that $|H_\xi|< \mu$.
\end{corollary}

\begin{corollary}  Let $\kappa > \omega$ be a regular cardinal such that $\kappa = \kappa^{<\kappa}$ and let $\lambda =  cf(2^\kappa)$. Let $X={^\kappa}\kappa$. Let $\eta, \mu, \kappa^+$ be cardinals such that $\kappa^+ \leqslant \mu < cf(2^\kappa)$. 
	If $X$ has no $(\eta, \mu, \alpha)$-calibre then: 
	\begin{itemize}
		\item [(i)] there exists a $\kappa^+$-strong sequence $(S_\xi, H_\xi)_{\xi < \lambda}$ in $X$ such that $|H_\xi|\leqslant \kappa$ for each $\xi < \lambda$; 
		\item [(ii)] there exists a $\kappa$-BK matrix in $X$;
		\item [(iii)]  $\mathfrak{l}_\kappa = \kappa^+$;
		\item [(iv)] There exists a $\kappa$-concentrated set of cardinality $\lambda$ in $X$;
		\item [(v)] there exists a $\kappa$-K-Lusin set of cardinality $\lambda$ in $X$.
	\end{itemize}
\end{corollary}
	
	\subsection{Boundedness,  dominating and strong sequences}
	
	Cardinal characteristics for $\kappa > \omega$  have been studied for various researchers and this topic is still vivid. In \cite{BBFM}, there are presented models in ZFC having the influence for generalized Cicho\'n diagram presented in Section 2. Other bibliographical data for this topic can be found  in \cite{KLLS}. Here we are concentrated on the number associated with strong sequences. Below we show some more interesting results concerning generalized boundedness and generalized dominating number.
	
In \cite{CS} it is proved that for a regular $\kappa> \omega$, $$\kappa^+ \leqslant cf(\mathfrak{b}_\kappa) = \mathfrak{b}_\kappa \leqslant cf(\mathfrak{d}_\kappa) \leqslant \mathfrak{d}_\kappa\leqslant 2^\kappa$$ and this is the only restriction on $\mathfrak{b}_\kappa$ and $\mathfrak{d}_\kappa$ that are provable in ZFC. \footnote{It shows that  $\mathfrak{b}_\kappa$ and $\mathfrak{d}_\kappa$ behave in the same way as $\mathfrak{b}$ and $\mathfrak{d}$.}
 
 Below, we present some connections between $\mathfrak{b}_\kappa$, $\mathfrak{d}_\kappa$ and $\hat{\mathfrak{s}}_\kappa$.
 
 1. In generalized Sacks forcing we have $\mathfrak{b}_\kappa=\kappa^+ <\mathfrak{d}_\kappa = \kappa^{++}$. Hence $\hat{\mathfrak{s}}_\kappa = \kappa^{++}$.
 
 2. If   $\mathfrak{b}_\kappa=\kappa^+ <\mathfrak{d}_\kappa = 2^\kappa$ then $\hat{\mathfrak{s}}_\kappa = 2^\kappa$.
 
 3. Let $\mathfrak{s}_\kappa$ and $\mathfrak{r}_\kappa$ means splitting and reaping number, respectively (see \cite{RS, RS1} for definitions). Then under GCH $\hat{\mathfrak{s}}_\kappa = 2^\kappa$ whenever $\mathfrak{s}_\kappa= \mathfrak{r}_\kappa = \kappa^+$.
 
 4. In the generalized Hechler forcing we have 
 $\mathfrak{r}_\kappa=\kappa^+ <\mathfrak{s}_\kappa = \kappa^{++}$ and hence $\hat{\mathfrak{s}}_\kappa = 2^\kappa$.
 
 5. If  $\mathfrak{b}_\kappa= \mathfrak{d}_\kappa = 2^\kappa = \kappa^{++}$ then $\hat{\mathfrak{s}}_\kappa < 2^\kappa$.
 
 6. In generalized Cohen forcing 
	$\mathfrak{b}_\kappa= \mathfrak{d}_\kappa  = \kappa^{+}$ then $\hat{\mathfrak{s}}_\kappa < 2^\kappa$.
	
7. If $\mathfrak{b}_\kappa= \mathfrak{d}_\kappa  = 2^\kappa$ then $\hat{\mathfrak{s}}_\kappa < 2^\kappa$.
	
8. If $\mathfrak{b}_\kappa= \mathfrak{d}_\kappa  = \kappa^{+}$ then $\hat{\mathfrak{s}}_\kappa < \kappa^{++}$.

9. Under $MA_\kappa+2^\kappa>\kappa^+$ $\hat{\mathfrak{s}}_\kappa < \kappa^{+}$.

10. If $\mathfrak{s}_\kappa= \kappa^+ <  \mathfrak{r}_\kappa = \kappa^{++}$ then $\hat{\mathfrak{s}}_\kappa < 2^\kappa$.

We can also give further conncections among $\hat{\mathfrak{s}}_\kappa$ and other cardinal invariants: $\mathfrak{a}_\kappa$ and $\mathfrak{i}_\kappa$ associated with generalize MAD families and generalized independent families, but it requires recalling here a number of definitions and is beyond the scope of this paper.

	\subsection{Partition relations and strong sequences}
	
	Let $\beta$ and $\eta$ be cardinals. By $\beta \ll \eta$ we denote: $\eta$ is \textit{$\beta$-strong inaccessible}, i. e. $\beta <\eta$ and $\alpha^\lambda <\eta$, whenever $\alpha < \eta, \lambda < \beta$.
\\
	
	Let $\alpha, \beta$ and $\lambda$ be cardinals and $n < \omega$.
	The \textit{arrow notation} 
	$$(\alpha) \to (\beta)^{n}_{\lambda}$$
	denotes the following partition relation: if
	$[\alpha]^n = \bigcup_{i < \alpha} P_i$,
	then there are $A \subset \alpha$ and $i < \lambda$ such that $|A| = \beta$, and $[A]^n \subset P_i$, (see e. g. \cite{CN}).
	\\
	
	In paper \cite{JJ2} there were proved the following results.

	\begin{theorem}[\cite{JJ2}]
		Let $\beta, \mu, \eta, \kappa$ be cardinals such that $\omega \leq \beta \ll \eta$,  $\mu < \beta$ and $\beta, \eta$ be regular.  Let $X$ be a set of cardinality $\eta$.
		The following statements are equivalent:
		\begin{itemize}
			\item[(i)]
			If in $X$ there exists a $\kappa$-strong sequences $(S_\alpha, H_\alpha)_{\alpha < \eta}$ with $|H_\alpha| \leq 2^\mu$ for all $\alpha < \eta$, then there exists a $\kappa$-strong sequence $(S_\alpha, T_\alpha)_{\alpha < \beta}$ with $T_\alpha \subseteq H_\alpha$ and $|T_\alpha|< \kappa$  for all $\alpha < \beta$ and $\kappa< \eta$.  
			\item [(ii)]  
			$(\eta) \to (\beta)^{2}_{\mu}.$
		\end{itemize} 
	\end{theorem}
	
	\begin{theorem} [\cite{JJ2}]
		Let $\beta, \kappa$ and $\eta$ be cardinals such that $\omega \leq \beta \ll \eta$ with $\eta$ - regular and $\beta$ - singular. Let $X$ be a set of cardinality $\eta$. 
		The following statements are equivalent.
		\begin{itemize}
			\item [(i)]
			If in $X$ there exists a $\kappa$-strong sequences $(S_\alpha, H_\alpha)_{\alpha < \eta}$ with $|H_\alpha| \leq 2^\beta$ for all $\alpha < \eta$, then there exists a $\kappa$-strong sequence $(S_\alpha, T_\alpha)_{\alpha < \beta}$ with $T_\alpha \subseteq H_\alpha$ and $|T_\alpha|< \kappa$ for all $\alpha < \beta^+$ and $\kappa < \eta$.  
			\item [(ii)] 
			$(\kappa) \to (\beta^+)^{2}_{\beta}.$
		\end{itemize}
	\end{theorem}
	
	Hence, based on these results one can conclude that there are connections between notions from Theorem 1 and partition relations. This will be considered in the next paper which is now in preparation.
	\\\\
	\noindent
	\textbf{\Large{Declarations}}
	\\
	
	\textbf{Data Availability} Data sharing not applicable to this article as no databases were generated or analyzed
	during the current study.
	
	\textbf{Conflict of Interest} The author has no conflict of interest to declare that are relevant to the content of this
	study.

	\textbf{Funding} This research received no specific grant from any funding agency in the public, commercial, or not-for-profit sectors.

	\end{document}